\documentclass[12pt]{scrartcl}
\usepackage[english]{babel}
\usepackage[utf8x]{inputenc}
\usepackage[T1]{fontenc}
\usepackage{lmodern}
\usepackage{amsfonts}
\usepackage{amsmath}
\usepackage{amssymb}
\usepackage{amsthm}
\usepackage{mathrsfs}
\usepackage[colorlinks=false,urlcolor=blue]{hyperref}
\usepackage{listings}
\usepackage{graphicx}
\usepackage[shortlabels]{enumitem}
\usepackage{mathtools}
\usepackage{xcolor}
\usepackage{comment}

\addtokomafont{section}{\rmfamily}
\addtokomafont{subsection}{\rmfamily}
\addtokomafont{subsubsection}{\rmfamily}
\addtokomafont{paragraph}{\rmfamily}
\swapnumbers
\theoremstyle{plain}
\newtheorem{thm}[paragraph]{Theorem}
\newtheorem{lem}[paragraph]{Lemma}
\newtheorem{cor}[paragraph]{Corollary}
\newtheorem{prop}[paragraph]{Proposition}
\theoremstyle{definition}
\newtheorem{defn}[paragraph]{Definition}
\newtheorem{rem}[paragraph]{Remark}
\newtheorem{ex}[paragraph]{Example}
\counterwithin{paragraph}{section}
\counterwithout{paragraph}{subsection}
\counterwithout{paragraph}{subsubsection}
\counterwithin{equation}{section}

\newcommand{\R}{\mathbb{R}}
\newcommand{\C}{\mathbb{C}}
\renewcommand{\H}{\mathbb{H}}
\newcommand{\N}{\mathbb{N}}
\newcommand{\Z}{\mathbb{Z}}
\newcommand{\Oct}{\mathbb{O}}
\renewcommand{\O}{\mathrm{O}}
\renewcommand{\P}{\mathbb{P}}
\newcommand{\Id}{\mathrm{Id}}
\newcommand{\vol}{\mathrm{vol}}
\newcommand{\Sym}{\operatorname{Sym}}
\newcommand{\tr}{\operatorname{tr}}
\newcommand{\im}{\operatorname{im}}
\newcommand{\diag}{\operatorname{diag}}
\newcommand{\End}{\operatorname{End}}
\newcommand{\Ad}{\operatorname{Ad}}
\newcommand{\ad}{\operatorname{ad}}
\newcommand{\U}{\operatorname{U}}
\newcommand{\SU}{\operatorname{SU}}
\newcommand{\su}{\mathfrak{su}}
\renewcommand{\u}{\mathfrak{u}}
\newcommand{\SO}{\operatorname{SO}}
\newcommand{\so}{\mathfrak{so}}
\newcommand{\Sp}{\operatorname{Sp}}
\renewcommand{\sp}{\mathfrak{sp}}
\newcommand{\Spin}{\mathrm{Spin}}
\newcommand{\spin}{\mathfrak{spin}}
\newcommand{\Cl}{\mathrm{C}\ell}
\makeatletter
\@ifundefined{sl}{%
}{
}
\makeatother
\newcommand{\spann}{\operatorname{span}}
\newcommand{\pr}{\operatorname{pr}}
\newcommand{\m}{\mathfrak{m}}
\newcommand{\h}{\mathfrak{h}}
\newcommand{\g}{\mathfrak{g}}
\renewcommand{\k}{\mathfrak{k}}
\renewcommand{\t}{\mathfrak{t}}
\newcommand{\intprod}{\mathbin{\lrcorner}}
\newcommand{\e}{\mathfrak{e}}
\newcommand{\f}{\mathfrak{f}}

\newcommand{\rmE}{\mathrm{E}}

\newcommand{\rmG}{\mathrm{G}}
\newcommand{\Hol}{\mathrm{Hol}}
\newcommand{\hol}{\mathfrak{hol}}
\newcommand{\stab}{\mathfrak{stab}}
\newcommand{\rank}{\operatorname{rk}}
\newcommand{\Curv}{\mathcal{K}}
\newcommand{\threead}{3-$(\alpha,\delta)$-Sasaki}
\newcommand{\W}{\mathcal{W}}
\DeclareMathOperator{\hodge}{\star}

\newcommand{\todo}[1]{\textcolor{red}{#1}}

\title{\rmfamily Geometries with parallel, skew-symmetric and closed torsion}
\author{Andrei Moroianu\footnote{Université Paris-Saclay, CNRS,  Laboratoire de mathématiques d'Orsay, 91405 Orsay, France.}\ \footnote{Institute of Mathematics ``Simion Stoilow'' of the Romanian Academy, 21 Calea Grivitei, 010702 Bucharest, Romania.}, Paul Schwahn\footnote{Universidade Estadual de Campinas, IMECC, Rua Sérgio Buarque de Holanda 651, 13083-859 Campinas-SP, Brazil.}}
\date{}

\begin{document}

\maketitle

\begin{abstract}
\noindent
We study Riemannian manifolds carrying a metric connection with parallel, skew-symmetric and closed torsion, which we call in short PSCT manifolds. We prove that PSCT manifolds always locally split into a product of well-understood factors, allowing a complete local classification. Further, we investigate various $G$-structures of PSCT type, with a focus on almost Hermitian structures and their possible Gray--Hervella classes.

\medskip


\medskip

\end{abstract}

\section{Introduction}
\label{sec:intro}

A \emph{geometry with parallel skew-symmetric torsion} $(M,g,\tau)$ is a Riemannian manifold together with a $3$-form $\tau\in\Omega^3(M)$ such that the connection $\nabla^\tau=\nabla^g+\tau$ parallelizes $\tau$, i.e.~$\nabla^\tau\tau=0$. This class of geometries has been widely studied (see e.g. \cite{CS,CMS,MS}), as it encompasses several important geometric situations: naturally reductive spaces, nearly Kähler manifolds in dimension $6$, nearly parallel $\rmG_2$-manifolds, Sasaki and \threead\ manifolds.

In contrast to the torsion-free case, geometries with parallel skew-symmetric torsion $\tau\neq0$ do not necessarily locally split into a product of irreducible manifolds. However, for every structure group $G$ that contains the holonomy group of $\nabla^\tau$ and leaves $\tau$ invariant, there is a canonical way to locally identify $(M,g,\tau)$ with the total space of a Riemannian submersion (which only depends on the structure of $TM$ as a $G$-representation) over a base that splits as a product of geometries with parallel skew-symmetric torsion, and whose fibers are naturally reductive Ambrose--Singer manifolds. This result is due to Cleyton--Moroianu--Semmelmann \cite{CMS} for $G=\Hol(\nabla^\tau)$, and was extended in \cite{MS} to every group $G$ containing $\Hol(\nabla^\tau)$ and stabilizing $\tau$, together with the fact that the individual factors of the base are \emph{stabilizer-irreducible} (i.e.~the stabilizer of the torsion acts irreducibly on the tangent space). These stabilizer-irreducible geometries are either Ambrose--Singer manifolds, nearly Kähler $6$-manifolds, nearly parallel $\rmG_2$-manifolds, or $3$-dimensional \cite[Thm.~4.4]{MS}.

Of separate interest are Riemannian manifolds carrying a metric connection with skew-symmetric and closed torsion, which are also called \emph{closed torsion geometries}. Stecker \cite{Ste} has shown that a remarkably similar local submersion result holds for them. If a manifold carries a $G$-structure parallelized by a metric connection with skew-symmetric and closed torsion, it is sometimes called a \emph{strong $G$-manifold with torsion}. They have ample applications in geometry and string theory; for example, they are sometimes used to construct solutions to the generalized Ricci flow \cite{genRic}. Strong $\U(n)$-manifolds with torsion (\emph{SKT manifolds}) are the setting of the related pluriclosed flow \cite{ST}, and strong $\SU(n)$-manifolds with torsion (\emph{SCYT manifolds}) are exactly the non-Kähler solitons of this flow. Strong $\rmG_2$-manifolds with torsion can be viewed as solutions to a limit of the heterotic $\rmG_2$-system \cite{G2T}.

In this article we consider the intersections of the two classes introduced above:

\begin{defn}
A \emph{PSCT (parallel, skew-symmetric and closed torsion) geometry} is a geometry with parallel skew-symmetric torsion $(M,g,\tau)$ that additionally satisfies $d\tau=0$.
\end{defn}

PSCT geometries were first studied by Agricola--Ferreira--Friedrich \cite{AFF} who proved a splitting theorem using the action of the Lie algebra bundle $\g_\tau\subset\so(TM)$ generated by all $X\intprod\tau$ where $X\in TM$. Moreover they showed that a \emph{Riemannian irreducible} PSCT geometry of dimension at least $5$ has to be a simple Lie group (with bi-invariant metric) or its noncompact symmetric dual. As we will see in \ref{sec:decomp}, it is in fact true that PSCT geometries always locally split into a product of \emph{stabilizer-irreducible} PSCT geometries, and the irreducible pieces are completely understood (see Theorems A and B below).

In the recent preprint \cite{Pap} Papadopoulos gives a synopsis of the proof in \cite{AFF} as well as a partial classification result for complete, simply connected PSCT geometries. More precisely, he classifies them up to a Riemannian reducible factor where the torsion form does not split accordingly. A PCST geometry $(M,g,\tau)$ where the torsion form $\tau$ respects any decomposition of $TM$ into $\nabla^g$-parallel subbundles is said to be of NSZ type. In our Corollary~\ref{corprod} we give the full classification of complete, simply connected PCST geometries, without the NSZ assumption.

The case of PSCT geometries parallelizing a Hermitian structure can be equivalently understood as follows: if $(M,g,J,\omega)$ is a Hermitian manifold, then it has a unique metric connection with skew-symmetric torsion, namely the \emph{Bismut connection}, with torsion $2\tau=-d^c\omega$. Its torsion is closed if and only if $dd^c\omega=0$, i.e.~if $(M,g,J,\omega)$ is pluriclosed. Therefore a PSCT geometry $(M,g,\tau)$ parallelizes a Hermitian structure $J$ with fundamental 2-form $\omega$ if and only if $(M,g,J,\omega)$ is pluriclosed and its Bismut connection has parallel torsion. Such Hermitian structures are also called \emph{Bismut Kähler-like}, and they have recently been studied and classified by
Barbaro--Pediconi--Tardini \cite{BPT}. Their examples show that even though a PSCT geometry always locally splits into ``boring'' irreducible PSCT geometries, a geometric structure parallelized by a PSCT connection does not necessarily respect this splitting, giving an interesting angle of study. Very recently, Barbaro--Pediconi \cite{BP} dropped the pluriclosed condition and replaced it by the assumption that the Bismut connection has parallel torsion and curvature. They call these geometries Bismut--Ambrose--Singer, and obtain their classification in the complete, simply connected case. 

For many $G$-structures of interest, a $G$-connection with skew-symmetric torsion $\nabla^\tau$ is unique if it exists (and then called \emph{characteristic connection}). This includes for example almost Hermitian structures, almost contact metric structures, $\rmG_2$-structures \cite{FI}, $\Spin(7)$-structures \cite{Spin7}, and $\Spin(9)$-structures \cite{Spin9}. In fact this is also true for $\Sp(q)\Sp(1)$-structures, as we show in~\ref{sec:unique}. In the above cases, where the $G$-structure determines the torsion form $\tau$, it thus makes sense to talk about \emph{$G$-structures with parallel skew-symmetric (closed) torsion}.

For the reader's convenience we list here the main results of the paper. We refer to the corresponding sections for the notation in the statements below. 

In~\ref{sec:decomp} we complete the proof of the decomposition of PSCT geometries into stabilizer-irreducible PSCT geometries which was first considered in \cite{AFF}:

\textbf{Theorem A.} (Corollary \ref{cor:split} below)
{\em Any PSCT geometry $(M,g,\tau)$ splits locally as a Riemannian product,
\[(M,g,\tau)=\prod_i(M_i,g_i,\tau_i),\qquad \tau=\sum_i\tau_i,\]
into PSCT factors which are $\stab(\tau_i)$-irreducible. If $(M,g)$ is complete and simply connected, the above decomposition holds globally.}

We then describe each stabilizer-irreducible factor using our recent results from \cite{MS}:

\textbf{Theorem B.} (Theorem \ref{thm:irred} below)
{\em Any $\stab(\tau)$-irreducible PSCT geometry with $\tau\neq 0$ is locally isomorphic to one of the following:
\begin{enumerate}[\upshape(i)]
 \item a simple compact Lie group $H$ with bi-invariant metric, where $\tau$ is a nonzero multiple of the canonical $3$-form, and $\hol(\nabla^\tau)$ is either trivial or $\h$,
 \item a flat Lie algebra $\h$, where $\tau$ is a nonzero multiple of the canonical $3$-form, and $\hol(\nabla^\tau)=\h$,
 \item an irreducible noncompact symmetric space $H^\C/H$, where $\tau$ is a nonzero multiple of the canonical $3$-form, and $\hol(\nabla^\tau)=\h$,
 \item a $3$-dimensional Riemannian manifold $(M,g)$, where $\tau$ is a nonzero multiple of $\vol_g$, and $\hol(\nabla^\tau)$ is either trivial, $\u(1)$ or $\so(3)$, according to whether $(M,g)$ is a round $3$-sphere with $\nabla^\tau$ the $(\pm)$-connection, a non-round $\alpha$-Sasaki manifold, or generic.
\end{enumerate}}

In the main part of the article (\ref{sec:almHerm} and \ref{sec:other}) we investigate which $n$-dimensional PSCT geometries parallelize in addition a $G$-structure, where $G\subset\SO(n)$. We do so in the cases where $G$ is an irreducible Riemannian holonomy group, or $\Spin(k)$ acting on $\R^n$ by its spin representation; in each case we fully classify the underlying PSCT geometries (up to understanding the local Riemannian factors).

We focus particularly on almost Hermitian structures ($G=\U(n/2)$), and obtain the following result which extends \cite[Thm.~A]{BPT} to the almost Hermitian setting. 

\textbf{Theorem C.} (Theorem \ref{thm:complex} below)
{\em Let $(M,g,\tau)$ be a PSCT geometry parallelizing a compatible almost complex structure $J$, i.e.~$\nabla^\tau J=0$. Then $(M,g,\tau)$ is locally isomorphic to a Riemannian product of the form
\[(H,g_H,\tau_H)\times\prod_{i=1}^s(S_i,g_i,\alpha_i\vol_i)\times (N,g_N,0),\]
where
\begin{itemize}
 \item $H$ is a Lie group with a bi-invariant metric $g_H$, compact Lie algebra $\h$, canonical three-form $\sigma$ and torsion form $\tau_H=-\frac12\sigma$, so that $\hol(\nabla^{\tau_H})=0$,
 \item for each $i$, $(S_i,g_i,\xi_i,\omega_i)$ are three-dimensional, non-round $\alpha_i$-Sasaki manifolds for some $\alpha_i\neq0$, with Reeb vector fields $\xi_i$, volume forms $\vol_i=\xi_i\wedge\omega_i$, and $\hol(\nabla^{\alpha_i\vol_i})=\u(1)$,
 \item $(N,g_N)$ is a Kähler manifold.
\end{itemize}
$J$ is determined by choosing transverse complex structures on $\xi_i^\perp\subset TS_i$, an almost complex structure on $\h\oplus\spann\{\xi_1,\ldots,\xi_r\}$, and the Kähler structure on $N$. If $(M,g)$ is complete and simply connected, the statement holds globally.}

We study next the possible Gray--Hervella classes \cite{GH} of PSCT geometries parallelizing a Hermitian structure. Using combinatorial results from the companion paper \cite{roots}, we classify PSCT geometries carrying a parallel semi-Kähler structure 
and we show that in dimension $n\geq6$, no other non-generic torsion class besides Hermitian and semi-Kähler occurs for PSCT almost Hermitian geometries with $\tau\neq0$:

\textbf{Theorem D.} (Theorems~\ref{thm:W1W3}, \ref{thm:W1W4} and \ref{thm:W3} below) 
{\em Let $(M^n,g)$ be a Riemannian manifold with $n\geq6$.
\begin{itemize}
\item If $(M,g)$ admits a PSCT almost Hermitian structure of class $\W_1\oplus\W_3$, then it is locally isomorphic to the product of a Kähler manifold with an even-dimensional Lie group whose Lie algebra is compact and different from $\su(2)\oplus\R$. Conversely, every such Riemannian product admits a PSCT almost Hermitian structure of class $\W_1\oplus\W_3$.
\item Every PSCT almost Hermitian structure on $(M,g)$ in the class $\W_1\oplus\W_4$ or $\W_3$ is automatically Kähler.
\end{itemize}}

 Note that every almost Hermitian PSCT geometry has automatically vanishing $\mathcal{W}_2$ component, so belongs to the $\mathcal{W}_1\oplus \mathcal{W}_3\oplus \mathcal{W}_4$ class. As mentioned above, the Gray--Hervella class $\W_3\oplus\W_4$ (corresponding to integrable Hermitian structures) was already studied in \cite[Thm.~A]{BPT}. Note also that the class $\W_1\oplus\W_3$, which appears in the first case of the above theorem, consists exactly of semi-Kähler almost Hermitian structures with skew-symmetric torsion.

Finally, in~\ref{sec:other} we study PSCT structures parallelizing further $G$-structures, and prove the following:

\textbf{Theorem E.} (Theorems~\ref{thm:parG}, \ref{thm:g2}, \ref{thm:spin7} and~\ref{thm:spink} below)
{\em
\begin{itemize}
\item Let $(M^n,g,\tau)$ be a simply connected PSCT geometry parallelizing a $G$-structure, where $G$ is either $\SU(k)$ with $n=2k\geq4$, or one of $\Sp(k)$, $\Sp(k)\Sp(1)$ with $n=4k\geq8$. Then $(M,g,\tau)$ decomposes as in \eqref{prod}, with $(M_i,g_i,\tau_i)$ being compact, simple Lie groups with flat $\nabla^{\tau_i}$. Moreover, if $G=\SU(k)$, then the irreducible torsion-free Riemannian factors are Calabi--Yau. If $G=\Sp(k)$ or $\Sp(k)\Sp(1)$ and $\tau\neq0$, then the irreducible torsion-free Riemannian factors are hyperkähler.
\item Every PSCT geometry $(M^7,g,\tau)$ with $\tau\neq0$ parallelizing a $\rmG_2$-structure is locally isomorphic to one of $S^3\times\R^4$, or $S^3\times S^3\times\R$, or $S^3\times X^4$, where $X^4$ is a $K3$ surface. In each of these cases, $S^3$ is equipped with a bi-invariant metric and flat Cartan--Schouten connection, while the other factors are torsion-free.
\item Let $(M^8,g,\tau)$ be a PSCT geometry with $\tau\neq0$, parallelizing a $\Spin(7)$-structure. Then it is locally isomorphic to one of $S^3\times\R^5$, or $S^3\times S^3\times\R^2$, or $\SU(3)$, or $S^3\times\R\times X^4$, where $X^4$ is a $K3$ surface. In each of these cases, $S^3$ resp.~$\SU(3)$ is equipped with a bi-invariant metric and flat Cartan--Schouten connection, while the other factors are torsion-free.
\item Let $(M^n,g,\tau)$ be a  PSCT geometry parallelizing a $\Spin(k)$-structure, where $k\geq8$. Then
\begin{enumerate}[\upshape(i)]
\item either $\tau=0$, $k=9$, and $(M,g)$ is locally isometric to the Cayley projective plane $\Oct\P^2$ or its noncompact dual $\Oct\H^2$, or
\item $\hol(\nabla^\tau)=0$ and $(M,g,\tau)$ is locally isomorphic to a compact Lie group with a left-invariant $\Spin(k)$-structure.
\end{enumerate}
\end{itemize}}

We note that in the cases $\SU(k)$ and $\Sp(k)$ this generalizes the results of Brienza--Fino--Grantcharov \cite[Thm.~1.1]{sCYT} and Brienza--Fino--Grantcharov--Verbitsky \cite[Thm.~3.5]{sHKT} to the nonintegrable setting. Also, in the cases $\rmG_2$ and $\Spin(7)$ we reproduce the results of Papadopoulos \cite[Thm.~3.1 and 3.2]{Pap}, removing the technical assumption of ``NSZ type''. In the particular case of $\rmG_2$, the spaces mentioned above have been extensively discussed by Fino--Fowdar \cite{FF}, who also showed that a PSCT $\rmG_2$-structure has flat characteristic connection under a genericity assumption on certain torsion forms \cite[Thm.~6.2]{FF}. It is notable that these are essentially the only known examples of strong $\rmG_2$-manifolds with (not necessarily parallel) torsion, and it is unclear if there are any more; as Fino--Martín-Merchán--Raffero \cite{FMR} proved, a compact \emph{homogeneous} strong $\rmG_2$-manifold with torsion must be either $S^3\times S^3\times S^1$ or $S^3\times T^4$.

\paragraph*{Acknowledgments.} 
A.M.~was partially supported by the PNRR-III-C9-2023-I8 grant \textbf{CF 149/31.07.2023} (Conformal Aspects of Geometry and Dynamics).

P.S.~was supported by FAPESP project \textbf{2024/08127-4}, part of the BRIDGES collaboration, and received partial support by FAPESP grant \textbf{2024/19272-5}.

\section{Preliminaries}
\label{sec:prelim}

First, we fix our notation and conventions. Whenever we consider tensors on a finite-dimensional Euclidean vector space $(V,g)$, we freely raise and lower indices with respect to the metric $g$. In particular we identify $V\cong V^\ast$, and any two-form $\alpha\in\Lambda^2V$ will be identified with a skew-symmetric endomorphism also denoted by $\alpha\in\so(V,g)$ via
\[g(\alpha(X),Y)=\alpha(X,Y),\qquad X,Y\in V.\]
For a skew-symmetric endomorphism $A\in\so(V,g)$, we denote with $A_\ast$ its Lie algebra action on tensors over $V$. On an exterior form $\alpha\in\Lambda^kV$, this can be written as
\begin{equation}
\label{a}
A_\ast\alpha=\sum_iAe_i\wedge(e_i\intprod\alpha),
\end{equation}
where $(e_i)$ is any orthonormal basis of $(V,g)$. If $\alpha,\beta\in\Lambda^2V$, then $\alpha_\ast\beta=[\alpha,\beta]$ is the commutator of endomorphisms.

Any three-form $\tau\in\Lambda^3V$ gives rise to a $(2,1)$-tensor $\tau\in V^\ast\otimes\so(V,g)$ by
\[g(\tau_XY,Z)=\tau(X,Y,Z),\qquad X,Y,Z\in V.\]
We denote with $\stab(\tau)\subset\so(n)$ be the Lie algebra of the stabilizer group of $\tau$ in $\O(n)$, that is
\[\stab(\tau):=\{\alpha\in\so(V,g)\,|\,\alpha_\ast\tau=0\}.\]
Moreover, let $\g_\tau\subset\so(V,g)$ be the Lie subalgebra generated by all the $\tau_X$ where $X\in V$, first considered in \cite{AF}. The \emph{kernel} of $\tau$ is the subspace
\[\ker\tau:=\{X\in V\,|\,\tau_X=0\}.\]
The \emph{Bianchi map} $b: \Sym^2\Lambda^2V\to\Lambda^4V$ is defined by
\[b(R)(X,Y,Z,W):=R(X,Y,Z,W)+R(Y,Z,X,W)+R(Z,X,Y,W)\]
for any $R\in\Sym^2\Lambda^2V$, $X,Y,Z,W\in V$, so that $\frac13b$ is the orthogonal projection to $\Lambda^4V$. For any three-form $\tau\in\Lambda^3V$, let $\tau^2\in\Sym^2\Lambda^2V$ be the $4$-tensor defined by
\begin{equation}
\label{tau2}
\tau^2_{X,Y}Z:=\tau_Z\tau_XY,\qquad X,Y,Z\in V.
\end{equation}
Then we have 
\begin{equation}
\label{xbtau2}
X\intprod b(\tau^2)=(\tau_X)_\ast\tau,
\end{equation}
and from \eqref{a} it follows that
\begin{equation}
\label{btau2}
b(\tau^2)=-\frac12\sum_i\tau_{e_i}\wedge\tau_{e_i}\in\Lambda^4V.
\end{equation}

If $(M,g,\tau)$ is a geometry with parallel skew-symmetric torsion, we can write the curvature of $\nabla^\tau$ as $R^\tau=R^g+\tau^2+b(\tau^2)$. Moreover, $d\tau=-4b(\tau^2)$, cf.~\cite{AFF,MS}. All these facts lead to the following equivalent characterizations of PSCT geometries.

\begin{lem}
Let $(M,g,\tau)$ be a geometry with parallel skew-symmetric torsion. Then the following conditions are equivalent:
\begin{enumerate}[\upshape(i)]
    \item $(M,g,\tau)$ is a PSCT geometry, i.e.~$d\tau=0$.
    \item $b(\tau^2)=0$.
    \item For any $p\in M$, $(T_pM,g_p,\tau_p)$ is a metric Lie algebra.
    \item $\nabla^g\tau=0$.
    \item The curvature of $\nabla^\tau$ satisfies the first Bianchi identity, i.e.~$b(R^\tau)=0$.
    \item For any $X\in TM$, $(\tau_X)_\ast\tau=0$.
    \item $\g_\tau\subset\stab(\tau)$.
\end{enumerate}
\end{lem}

Lastly, we review a few well-known geometries that will become relevant in the sequel.

\begin{defn}\label{a-sas}
Let $\alpha\in\R^\times$. An \emph{$\alpha$-Sasaki manifold} $(M, g, \xi, \Phi)$ is a Riemannian manifold $(M^{2n+1},g)$ together with a unit length Killing vector field $\xi$ and a skew-symmetric endomorphism $\Phi$ satisfying
\begin{equation}\label{das}
   \Phi^2 = -\Id + \xi \otimes \xi, \quad  \nabla^g_X\xi = \alpha \Phi(X), \quad
        \nabla^g_X\Phi = \alpha \xi \wedge X,\qquad\forall X\in T M.
\end{equation}
\end{defn}
Every $\alpha$-Sasaki manifold $(M, g, \xi, \Phi)$ has a \emph{characteristic connection} $\nabla^\tau$ with $\tau:=\alpha\xi\wedge\Phi$ that parallelizes the vector field $\xi$ and the endomorphism $\Phi$. In particular $\nabla^\tau\tau=0$, so $(M,g,\tau)$ is a geometry with parallel skew-symmetric torsion.

A Sasaki manifold is nothing but a $1$-Sasaki manifold. If $(M,g,\xi,\Phi)$ is $\alpha$-Sasaki, then $(M,\alpha^2g,\alpha^{-1}\xi,\Phi)$ is a Sasaki manifold.

\begin{ex}\label{as}
Let $(M,g)$ be an oriented 3-dimensional Riemannian manifold and $\tau:=\alpha\vol_g$ for some $\alpha\in\R^\times$. Then $(M,g,\tau)$ is a PSCT geometry. A unit length Killing vector field $\xi$ is $\nabla^\tau$-parallel if and only if $(M,g,\xi,\Phi:=\hodge\xi)$ is $\alpha$-Sasaki. Moreover $\nabla^\tau$ is flat if and only if $(M,g)$ is locally isometric to the round sphere of radius $1/\alpha$, cf.~\cite[Prop.~4.12, Rem.~4.13]{MS}.
\end{ex}

\begin{ex}\label{4dim}
Let $(M,g,\tau)$ be an oriented 4-dimensional geometry with parallel skew-symmetric torsion. Then it is automatically PSCT: the vector field $\theta:=\hodge\tau$ is $\nabla^\tau$-parallel and satisfies $\tau_X\theta=0$ for every $X\in TM$, since
\[\theta\intprod\tau=-\theta\intprod\hodge\theta=\hodge(\theta\wedge\theta)=0.\]
Thus $\theta$ is $\nabla^g$-parallel, so the same holds for $\tau$.
\end{ex}

\begin{ex}\label{liegroups}
Let $G$ be a Lie group with Lie algebra $\g$, endowed with a bi-invariant metric $g$. This entails that $\g$ is compact, i.e.~a product of compact simple Lie algebras and an abelian one.

The Lie bracket together with the metric defines a three-form $\sigma\in\Omega^3(G)$ by
\[\sigma(X,Y,Z):=g([X,Y],Z)\]
for left-invariant vector fields $X,Y,Z\in\g$ that we call the \emph{canonical three-form}. There is a distinguished one-parameter family $(\nabla^t)_{t\in\R}$ of invariant metric connections on $G$, called the \emph{Cartan--Schouten family}, defined by
\[\nabla^t:=\nabla^g+\frac{t}{2}\sigma.\]
These connections all parallelize $\sigma$, and the Jacobi identity implies that $b(\sigma^2)=0$, so $(G,g,t\sigma)$ is a PSCT geometry for any $t\in\R$. The connection $\nabla^t$ is related to $\nabla^{-t}$ by an isometric anti-automorphism of $G$, namely the map $x\mapsto x^{-1}$. If $t=\pm1$, then the connections $\nabla^t$ are flat, and called the \emph{$(\pm)$-connections}. Since $\nabla^g_XY=\frac12[X,Y]$ for left-invariant vector fields $X,Y\in\g$, it follows that the $(-)$-connection parallelizes left-invariant vector fields; similarly, the $(+)$-connection parallelizes right-invariant vector fields.
\end{ex}

\begin{ex}
Let again $G$ be a Lie group with compact Lie algebra $\g$ and bi-invariant metric $g$. Then the non-compact dual $G^\C/G$ of the symmetric space $G=\frac{G\times G}{\diag(G)}$ also carries an invariant canonical three-form and hence one can define an analogue of the Cartan--Schouten family on $G^\C/G$. The same goes for the Lie algebra $\g$ itself, with a flat invariant metric, cf.~\cite[Ex.~4.2]{AFF}. All of these connections define PSCT geometries, but except for the Levi-Civita connection on $\g$, the connection $\nabla^t$ is never flat.
\end{ex}

\section{A decomposition theorem}
\label{sec:decomp}

In this section we show that any PSCT geometry locally splits into a product of known geometries. This splitting result is morally already contained in \cite{AFF}; however, the authors do not justify why $\g_\tau$-invariant subbundles of $TM$ are parallel, and they only reduce the problem to the study of $\g_\tau$-irreducible PSCT geometries, leaving open what exactly those are. With the results of \cite{MS}, this question can now be answered completely. 

\begin{thm}
\label{thm:parallel_dist}
Let $(M,g,\tau)$ be a PSCT geometry with $\ker\tau=0$, and let $TM=\bigoplus_iT_i$ be a decomposition into $\g_\tau$-irreducible distributions. Then the distributions $T_i$ are parallel under both $\nabla^g$ and $\nabla^\tau$.
\end{thm}
\begin{proof}
We first remark that $\tau\in\bigoplus_i\Lambda^3T_i$ and that the distributions $T_i$ are mutually orthogonal by \cite[Prop.~4.2]{AF}. If we denote by $\tau_i$ the component of $\tau$ in $\Lambda^3T_i$, then one can write $\g_\tau=\bigoplus_i\g_{\tau_i}$ where each $\g_{\tau_i}$ is a Lie subalgebra of $\so(T_i)$ and $\ker\g_{\tau_i}|_{T_i}=0$. Moreover, any two $T_i$ and $T_j$, $i\neq j$, are inequivalent representations  of $\g_{\tau}$, since $\g_{\tau_i}\subset\g_\tau$ acts trivially on $T_j$ and nontrivially on $T_i$.

For each $i$, the $\g_\tau$-invariance of the distribution $T_i$ directly implies that it is $\nabla^g$-parallel if and only if it is $\nabla^\tau$-parallel. We will show that this last property holds.

Fix any point $x\in M$ and loop $\gamma$ based at $x$, and let $\varphi:T_xM\to T_xM$ be the parallel transport along $\gamma$ with respect to $\nabla^\tau$. The $\g_\tau$-invariant decomposition $T_xM=\bigoplus_i T_i$ gives rise to an $\Ad_\varphi(\g_\tau)$-invariant decomposition $T_xM=\bigoplus_i \varphi(T_i)$. But $\tau$ is $\nabla^\tau$-parallel, so $\Ad_\varphi(\g_\tau)=\g_\tau$. As we have seen above, the representations $T_i$ are pairwise inequivalent, so any splitting of $T_xM$ into $g_\tau$-irreducible representations is unique (up to reordering the summands).

By continuity, this shows that $\varphi(T_i)=T_i$ for every $i$, so the distributions $T_i$ are indeed $\nabla^\tau$-parallel.
\end{proof}

\begin{cor}
\label{cor:split}
Any PSCT geometry $(M,g,\tau)$ splits locally as a Riemannian product,
\[(M,g,\tau)=\prod_i(M_i,g_i,\tau_i),\qquad \tau=\sum_i\tau_i,\]
into PSCT factors which are $\stab(\tau_i)$-irreducible. If $(M,g)$ is complete and simply connected, the above decomposition holds globally.
\end{cor}
\begin{proof}
First, we recall that $\ker\tau\oplus(\ker\tau)^\perp$ is a splitting into $\nabla^g$-parallel subbundles \cite[Thm.~3.4]{AFF}. With de Rham's theorem, we can therefore write locally $(M,g,\tau)=(M_0,g_0,0)\times(M',g',\tau)$, and the first factor is of course irreducible under $\stab(0)=\so(\dim M_0)$. We can now apply Theorem~\ref{thm:parallel_dist} to $(M',g',\tau)$, and again with de Rham's theorem we obtain locally a splitting $(M',g',\tau)=\prod_i(M_i,g_i,\tau_i)$ into pieces where $\g_{\tau}$, and hence $\g_{\tau_i}$, acts irreducibly on $TM_i$. But $\g_{\tau_i}\subseteq\stab(\tau_i)$, so the $(M_i,g_i,\tau_i)$ are indeed $\stab(\tau_i)$-irreducible.
All the above decompositions are global if $(M,g)$ is complete and simply connected.
\end{proof}

By going through the list of $\stab(\tau)$-irreducible geometries with parallel skew-symmetric torsion in \cite{MS}, we can single out the possible $\stab(\tau)$-irreducible PSCT geometries:

\begin{thm}
\label{thm:irred}
Any $\stab(\tau)$-irreducible PSCT geometry with $\tau\neq 0$ is locally isomorphic to one of the following:
\begin{enumerate}[\upshape(i)]
 \item a simple compact Lie group $H$ with bi-invariant metric, where $\tau$ is a nonzero multiple of the canonical $3$-form, and $\hol(\nabla^\tau)$ is either trivial or $\h$,
 \item a flat Lie algebra $\h$, where $\tau$ is a nonzero multiple of the canonical $3$-form, and $\hol(\nabla^\tau)=\h$,
 \item an irreducible noncompact symmetric space $H^\C/H$, where $\tau$ is a nonzero multiple of the canonical $3$-form, and $\hol(\nabla^\tau)=\h$,
 \item a $3$-dimensional Riemannian manifold $(M,g)$, where $\tau$ is a nonzero multiple of $\vol_g$, and $\hol(\nabla^\tau)$ is either trivial, $\u(1)$ or $\so(3)$, according to whether $(M,g)$ is a round $3$-sphere with $\nabla^\tau$ the $(\pm)$-connection, a non-round $\alpha$-Sasaki manifold, or generic.
\end{enumerate}
\end{thm}
\begin{proof}
We apply \cite[Thm.~4.4]{MS} to $\mathfrak{h}:=\stab(\tau)$, which contains $\hol(\nabla^\tau)$ since $\nabla^\tau\tau=0$, and consider the following dichotomy:
\begin{enumerate}[(1)]
\item The Berger algebra $\Curv(\h)$ is trivial. Then $(M,g)$ is locally isometric to a non-symmetric isotropy irreducible homogeneous space $G/H$ with $\h=\mathrm{Lie}(H)$, and $\nabla^\tau$ is the canonical connection, cf.~\cite[\S2.4]{MS}. In this case, if $\g=\h\oplus\m$ is a reductive decomposition of $\g=\mathrm{Lie}(G)$, $\m$ can be identified with the tangent space at some point and $\tau(X,Y)=-\tfrac12[X,Y]_\m$ for every $X,Y\in\m$. The condition $b(\tau^2)=0$, together with the fact that $[\h,\m]\subset\m$ and the Jacobi identity in $\g$, gives for every $X,Y,Z\in \m$:
\begin{align*}
0&=\tau_Z\tau_XY+\tau_X\tau_YZ+\tau_Y\tau_ZX\\
&=\frac{1}{4}([Z,[X,Y]_\m]_\m+[X,[Y,Z]_\m]_\m+[Y,[Z,X]_\m]_\m)\\
&=-\frac{1}{4}([Z,[X,Y]_\h]+[X,[Y,Z]_\h]+[Y,[Z,X]_\h]).
\end{align*}
However this means that the map $\pr_\h: \Lambda^2\m\to\Lambda^2\m$ belongs to $\Curv(\h)$, which is a contradiction to $\Curv(\h)=0$.
\item $\Curv(\h)\neq0$. By the proof of \cite[Thm.~4.4]{MS}, this together with $b(\tau^2)=0$ implies that $(M,g)$ is either three-dimensional with $\tau$ a nonzero multiple of $\vol_g$, or locally isometric to one of the irreducible symmetric spaces $(H\times H)/H$ or $H^\C/H$ with $\h=\mathrm{Lie}(H)$, or to the Euclidean vector space $\h$, and $\tau$ is any nonzero multiple of the canonical 3-form on $\h$. This gives cases (iv) resp.~(i)--(iii) of the theorem statement.

In cases (i)--(iii), the condition $b(\tau^2)=0$ is just the Jacobi identity of $\h$, and $\hol(\nabla^\tau)$ is either $0$ or $\h$ by \cite[Prop.~4.5]{MS}, with the former only possible in case (i).

Finally, in case (iv), the condition $b(\tau^2)=0$ is vacuous because any four-form vanishes on a three-manifold. The characterization of $\hol(\nabla^\tau)$ when $\tau=\alpha\vol_g$ for some $\alpha\in\R^\times$ follows from Remark \ref{as}, cf.~\cite[Prop.~4.12 and Rem.~4.13]{MS}.
\end{enumerate}
\end{proof}

\begin{rem}\label{overlap}
In case (i), the holonomy algebra is trivial if and only if $\nabla^\tau$ is one of the $(\pm)$-connections from Example~\ref{liegroups}. Moreover, there is a small overlap between the entries (i) and (iv): the round $S^3$ with its Cartan--Schouten connections appears in both cases.
\end{rem}

We can summarize our findings as follows:

\begin{cor}\label{corprod}
Every PSCT geometry $(M,g,\tau)$ is locally isomorphic to a Riemannian product of the form
\begin{equation}
\prod_{i=1}^r(M_i,g_i,\tau_i)\times(\R^k,g_{\mathrm{Eucl}},0)\times\prod_{j=1}^s(N_j,h_j,0)
\label{prod}
\end{equation}
where $(M_i,g_i,\tau_i)$ are spaces from Theorem~\ref{thm:irred}, $(\R^k,g_{\mathrm{Eucl}})$ is flat, and $(N_j,h_j)$ are irreducible, non-flat Riemannian manifolds. The torsion $\tau=\sum_i\tau_i$ is only supported on the factors $M_i$ (with $\tau_i\neq0$ for every $i$), and the holonomy algebra is
\[\hol(\nabla^\tau)=\diag(\hol(\nabla^{\tau_1}),\ldots,\hol(\nabla^{\tau_r}),\underbrace{0,\ldots,0}_{k},\hol(\nabla^{h_1}),\ldots,\hol(\nabla^{h_s})).\]
If $(M,g,\tau)$ is complete and simply connected, then it is globally isomorphic to~\eqref{prod}.
\end{cor}

\begin{rem}
When $\ker\tau=0$, we observe from Theorem~\ref{thm:irred} that actually $\g_\tau=\stab(\tau)$. This may be seen a priori: each tangent space is a metric, thus compact, Lie algebra isomorphic to $\g_\tau$, with Lie bracket given by $\tau$. If $\ker\tau=0$, this Lie algebra is semisimple. The Lie algebra $\stab(\tau)$ corresponds to the skew-symmetric derivations of this Lie algebra, and $\g_\tau=\stab(\tau)$ follows from the fact that all derivations of a semisimple Lie algebra are inner.
\end{rem}

\section{PSCT almost Hermitian structures}
\label{sec:almHerm}

In this section we study PSCT geometries for which the connection with torsion parallelizes an almost Hermitian structure. First, we state an algebraic lemma that we will use in this and the following section.

\begin{lem}
\label{lem:reducedstructure}
Let $\h$ be a compact Lie algebra, $V$ a nontrivial irreducible real representation of $\h$, and $r\in\N$ a non-negative integer. Consider the sum $V\oplus\R^r$ of $V$ with a trivial representation.
\begin{enumerate}[\upshape(i)]
    \item If $\dim(V\oplus\R^r)=2n$ for some $n\ge 1$ and $\h$ preserves an orthogonal complex structure (i.e.~a $\U(n)$-structure) on $V\oplus\R^r$, then it preserves one on $V$.
    \item If $\dim(V\oplus\R^r)=4n$ with $n\geq2$, $r\neq0$, and $\h$ preserves an orthogonal quaternionic structure (i.e.~an $\Sp(n)\Sp(1)$-structure) on $V\oplus\R^r$, then $\dim(V)=4q$ for some $q\in\N$ and $\h$ preserves an orthogonal hypercomplex structure (i.e.~an $\Sp(q)$-structure) on $V$.
    \item The adjoint representation of a compact simple Lie algebra $\h$ does not carry an $\h$-invariant $\U(n)$- or $\Sp(n)\Sp(1)$-structure.
\end{enumerate}
\end{lem}
\begin{proof}
\begin{enumerate}[\upshape(i)]
\item An orthogonal complex structure on $V\oplus\R^r$ is an invertible endomorphism contained in
\[\so(V\oplus\R^r)\cong\so(V)\oplus\so(\R^r)\oplus(V\otimes\R^r).\]
Since $V$ is nontrivial and irreducible, the space of invariant elements is
\[\so(V\oplus\R^r)^\h=\so(V)^\h\oplus\so(\R^r).\]
This contains an invertible endomorphism if and only if $\so(V)^\h\neq0$; in this case, every invariant complex structure preserves the decomposition $V\oplus\R^r$ and hence restricts to an invariant complex structure on $V$.
\item Assume first that $\dim V\geq4$. Again, $V$ being irreducible and nontrivial, any three-dimensional invariant subspace of $\so(V\oplus\R^r)$ must be contained in $\so(V)\oplus\so(\R^r)$. If this invariant subspace is spanned by complex structures, which are invertible, and if $r\neq0$ then $\h$ must be acting trivially on it, since it must have nonzero projections to both $\so(V)$ and the trivial $\so(\R^r)$. The projection to $\so(V)$ must then be a three-dimensional subspace of $\h$-invariant complex structures on $V$, with a basis satisfying the quaternionic relations, i.e.~a hypercomplex structure on $V$.

If $\dim V=3$ (in which case $\h\cong V\cong\su(2)$), then apart from $\so(\R^r)$, which does not contain invertible elements, the only possible three-dimensional invariant subspaces of $\so(\h\oplus\R^r)$ are of the form
\[\{\alpha\ad(X)+X\wedge v\,|\,X\in\h\}\]
for $\alpha\in\R$ and $v\in\R^r$. Moreover, $r\geq5$, since by assumption, $\dim(\h\oplus\R^r)\geq8$. Then all elements of the above subspace are not invertible since their kernel contains the orthogonal complement of $v$ in $\R^r$, which is at least 4-dimensional. Thus $\h$ does not preserve any orthogonal quaternionic structure on $V\oplus\R^r$ when $\dim V=3$.

Finally, assume $\dim V=2$. Then necessarily $\h\cong\u(1)$. An invariant orthogonal quaternionic structure on $V\oplus\R^r\cong\R^{4n}$ then gives rise to an injective homomorphism $\h\hookrightarrow\sp(n)\oplus\sp(1)$. With respect to some Cartan subalgebra $\t$ of $\sp(n)\oplus\sp(1)$ containing the image of this homomorphism, the weights of the standard representation $\C^{2n}=\R^{4n}$ of $\sp(n)\oplus\sp(1)$ are
\[e_1\pm e_{n+1},\ e_2\pm e_{n+1},\ \ldots,\ e_n\pm e_{n+1}\]
for a suitable basis $(e_1,\ldots,e_{n+1})$ of the weight lattice. Since $\R^{4n}=V\oplus\R^r$ as a $\h$-representation, this means that $\h$ is mapped to a one-dimensional subspace of $\t$ annihilating all but one of the weights above. But this is impossible for $n\geq2$, so again, $\h$ does not preserve any orthogonal quaternionic structure on $V\oplus\R^r$ when $\dim V=2$.
\item If $\h$ is simple and $V=\h$, then $\so(\h)$ further decomposes into $\h\oplus\m$ with $\m$ irreducible (see e.g.~\cite[Prop.~7.49]{besse}). Note that $\dim\m=(\dim\h)(\dim\h-3)/2$, so in particular $\dim\m\neq1,3$, thus $\h$ cannot admit an invariant complex or quaternionic structure.
\end{enumerate}
\end{proof}

We are now ready for stating the general structure theorem of almost Hermitian PSCT geometries.

\begin{thm}
\label{thm:complex}
Let $(M,g,\tau)$ be a PSCT geometry parallelizing an almost Hermitian structure $J$, i.e.~$\nabla^\tau J=0$. Then $(M,g,\tau)$ is locally isomorphic to a Riemannian product of the form
\[(H,g_H,\tau_H)\times\prod_{i=1}^s(S_i,g_i,\alpha_i\vol_i)\times (N,g_N,0),\]
where
\begin{itemize}
 \item $H$ is a Lie group with a bi-invariant metric $g_H$, compact Lie algebra $\h$, canonical three-form $\sigma$ and torsion form $\tau_H=-\frac12\sigma$, so that $\hol(\nabla^{\tau_H})=0$,
 \item for each $i$, $(S_i,g_i,\xi_i,\omega_i)$ are three-dimensional, non-round $\alpha_i$-Sasaki manifolds for some $\alpha_i\neq0$, with Reeb vector fields $\xi_i$, volume forms $\vol_i=\xi_i\wedge\omega_i$, and $\hol(\nabla^{\alpha_i\vol_i})=\u(1)$,
 \item $(N,g_N)$ is a Kähler manifold.
\end{itemize}
$J$ is determined by choosing transverse complex structures on $\xi_i^\perp\subset TS_i$, an almost complex structure on $\h\oplus\spann\{\xi_1,\ldots,\xi_r\}$, and the Kähler structure on $N$. If $(M,g)$ is complete and simply connected, the statement holds globally.
\end{thm}
\begin{proof}
By Corollary~\ref{corprod}, $(M,g,\tau)$ can be locally decomposed as in \eqref{prod}. Suppose $(M,g,\tau)$ has a local irreducible factor $(M_i,g_i,\tau_i)$ as in Theorem~\ref{thm:irred}, and assume that $\hol(\nabla^{\tau_i})\neq0$. Then either $\hol(\nabla^{\tau_i})$ is a compact simple Lie algebra whose holonomy representation $T_pM_i$ is the adjoint representation, or $\hol(\nabla^{\tau_i})=\u(1)$ and $(M_i,g_i,\tau_i)$ is a non-round $\alpha_i$-Sasaki manifold for some $\alpha_i\in\R^\times$.

The first case cannot occur, since by Lemma~\ref{lem:reducedstructure}, $\hol(\nabla^{\tau_i})$ does not preserve a $\U(n)$-structure on $T_pM$. For the three-dimensional $\alpha_i$-Sasaki factors, we have $T_pM_i=\R\xi_i\oplus\xi_i^\perp$, and the almost complex structure $J$ must preserve the subspace $\xi_i^\perp$ in order to be $\hol(\nabla^{\tau_i})$-invariant.

Lastly, if $\hol(\nabla^{\tau_i})=0$, then $(M_i,g_i,\tau_i)$ belongs to case (i) of Theorem~\ref{thm:irred} (see also Remark~\ref{overlap}), with $\nabla^{\tau_i}$ being one of the $(\pm)$-connections; since they are related by an isometry, we can assume it is the $(-)$-connection.

Consider now the Riemannian factors $(N_j,h_j)$ in~\eqref{prod}. Applying Lemma~\ref{lem:reducedstructure} for the representation of $\hol(\nabla^{h_j})$ on $T_pM$ shows that $(N_j,h_j)$ are irreducible Kähler manifolds.

The statement now follows by defining $H$ as the product of all factors of type (i) together with the flat factor, and $N$ as the product of the $N_j$.
\end{proof}

Almost Hermitian structures with skew-symmetric torsion are exactly those of Gray--Hervella type $\W_1\oplus\W_3\oplus\W_4$ \cite{GH}. As mentioned in the introduction, the connection with skew-symmetric torsion preserving the almost Hermitian structure is unique. We now investigate the possible torsion classes that an PSCT almost Hermitian structure can have.

The case where $J$ is integrable (class $\W_3\oplus\W_4$) has been investigated in \cite{BPT}. An integrable PSCT Hermitian structure is also called \emph{Bismut Kähler-like} since the curvature $R^\tau$ of its Bismut connection has the same symmetries as the Riemannian curvature tensor of a Kähler metric. Every PSCT Hermitian manifold is as in Theorem~\ref{thm:complex}, together with the additional constraint that there exists a Cartan subalgebra $\t\subset\h$ such that $J$ preserves the distribution spanned by $\t\oplus\spann\{\xi_1,\ldots,\xi_r\}$ and projects to an invariant complex structure on the full flag manifold $H/T$, where $T=\exp\t$. 

Before we turn to the other torsion classes, we note that the Gray--Hervella classification degenerates in dimension four; we study this case first.

\begin{prop}
Let $(M,g,\tau)$ be a $4$-dimensional geometry with parallel skew-sym\-metric torsion, $\tau\neq0$. Then $\nabla^\tau$ parallelizes an almost Hermitian structure $J$ if and only if $(M,g,J)$ is a Vaisman manifold.
\end{prop}
\begin{proof}
First, recall that $(M,g,\tau)$ is PSCT (Example~\ref{4dim}), and let $\theta:=\hodge\tau$. If $J$ is an almost Hermitian structure, we have for every $X\in TM$
\begin{align*}
(\tau_X)_\ast J&=-[X\intprod\hodge\theta,J]=[\hodge(X\wedge\theta),J]=-J_\ast(\hodge(X\wedge\theta))\\
&=-\hodge(JX\wedge \theta+X\wedge J\theta)=-(JX\wedge\theta+X\wedge J\theta),
\end{align*}
since $(0,2)+(2,0)$ forms are self-dual in dimension 4. Therefore the condition $\nabla^\tau J=0$ is equivalent to 
\[\nabla^g_XJ=JX\wedge \theta+X\wedge J\theta,\qquad\forall\ X\in TM,\]
which means that $(M,g,J)$ is locally conformally Kähler with Lee form $\theta$. Since moreover $\nabla^g\theta=\nabla^\tau\theta=0$, the statement follows.
\end{proof}

From now on we will thus restrict our attention to the case $\dim(M)\geq6$.

\begin{thm}
\label{thm:W1W3}
Every Riemannian manifold $(M^n,g)$ with $n\geq 6$ which admits a PSCT almost Hermitian structure of class $\W_1\oplus\W_3$ (skew-symmetric torsion and semi-Kähler) is locally isomorphic to the product of a Kähler manifold with an even-dimensional Lie group whose Lie algebra is compact and different from $\su(2)\oplus\R$. Conversely, every such Riemannian product admits a PSCT almost Hermitian structure of class $\W_1\oplus\W_3$.
\end{thm}
\begin{proof}
The class $\W_1\oplus\W_3$ consists of almost Hermitian structures $(g,J,\omega)$ which are parallel with respect to a connection $\nabla^\tau=\nabla^g+\tau$ with skew-symmetric torsion, and which are in addition \emph{semi-Kähler}. This last condition means $d^\ast\omega=0$, or equivalently $\omega\intprod\tau=0$. Indeed, since $\omega$ is $\nabla^\tau$-parallel, we have
\begin{equation}
d^\ast\omega=-\sum_ie_i\intprod\nabla^g_{e_i}\omega=\sum_ie_i\intprod(\tau_{e_i})_\ast\omega=\sum_i[\tau_{e_i},J]e_i=\sum_i\tau_{e_i}Je_i=2\omega\intprod\tau.
\label{codiff}
\end{equation}
Assume now that $(M,g,\tau)$ is PSCT. By Theorem~\ref{thm:complex}, the torsion form is given by 
\[\tau=-\frac12\sigma+\sum_j\alpha_j\vol_j,\]
where $\sigma$ is the canonical three-form of the Lie group factor $H$ and $\vol_j$ is the volume form of the $\alpha_j$-Sasaki factor $S_j$. The almost complex structure $J$ preserves $T_0:=\h\oplus\spann\{\xi_j\}_j$ and each of the $\xi_j^\perp\subset TS_j$. Then
\[\omega\intprod\tau=-\frac12\omega_0\intprod\sigma+\sum_j\alpha_j\xi_j,\]
where $\omega_0$ is the fundamental $2$-form of $J|_{T_0}$. Since each $\alpha_i$ is non-zero, $\omega\intprod\tau$ vanishes if and only if there are no $\alpha$-Sasaki factors and $\omega_0\intprod\sigma=0$.

Thus, Theorem~\ref{thm:complex} shows that if $(M,g,J)$ is a PSCT almost Hermitian structure of class $\W_1\oplus\W_3$, then $(M,g)$ is locally the Riemannian product of a Kähler manifold with a Lie group $H$ (with compact Lie algebra $\h$), carrying a bi-invariant metric and a left-invariant almost Hermitian structure $J_0$ (since $J_0$ is parallel for the $(-)$-connection) satisfying $\omega_0\intprod\sigma=0$. Since in dimension 4 the Lefschetz operator $\omega_0\intprod:\Lambda^3\h\to\h$ of any almost Hermitian form $\omega_0$ is an isomorphism, and the canonical three-form $\sigma$ of $\su(2)\oplus\R$ is non-zero, this last condition shows that $\h$ is different from $\su(2)\oplus\R$.

We will show that, conversely, every even-dimensional Riemannian Lie group $(H,g)$ whose Lie algebra $\h$ is compact and different from $\su(2)\oplus\R$ admits a left-invariant PSCT almost Hermitian structure of class $\W_1\oplus\W_3$, in other words, an almost Hermitian structure $J$ on $\h$ such that $\omega\intprod\sigma=0$. Fix a Cartan subalgebra $\t\subset\h$ and a Weyl chamber. Let $R^+$ denote the set of positive roots with respect to this Weyl chamber, and let $\h_\alpha$ denote the real root spaces of $\h$, so that $\h=\t\oplus\bigoplus_{\alpha>0}\h_\alpha$.  We then have for each root $\alpha$
\[[t,X]=\alpha(t)J_\alpha X\]
for all $t\in\t$, $X\in\h_\alpha$, where $J_\alpha$ is a complex structure on $\h_\alpha$ compatible with the metric $g$. The canonical three-form can therefore be written as
\[\sigma=\sum_{\alpha\in R^+}\alpha\wedge J_\alpha+\sigma_0,\]
with $\sigma_0\in\Lambda^3\bigoplus_{\alpha>0}\h_\alpha$.

Consider a subset $S\subset R^+$ with the following two properties:
\begin{enumerate}[\upshape(i)]
    \item $S$ is balanced, i.e.~there exist signs $(s_\alpha)_{\alpha\in S}$ such that $\sum_{\alpha\in S}s_\alpha\alpha=0$,
    \item $R^+\setminus S$ is strongly orthogonal, that is, for any $\alpha,\beta\in R^+\setminus S$, neither $\alpha+\beta$ nor $\alpha-\beta$ is a root.
\end{enumerate}
Such subsets of $R^+$ are called \emph{well-balanced} and are studied in our companion article \cite{roots}. If $S\subset R^+$ is well-balanced, we can define an almost Hermitian structure $J$ on $\h$ by
\[J:=J_\t+\sum_{\alpha\in S}s_\alpha J_\alpha+J',\]
where $J_\t$ is any almost Hermitian structure on $\t$ (note that $\t$ has even dimension), and $J'$ is any almost Hermitian structure on $\bigoplus_{\alpha\in R^+\setminus S}\h_{\alpha}$ with $J'\h_\alpha\subset\h_\alpha^\perp$ for all $\alpha\in R^+\setminus S$. The latter condition is always satisfiable if $|R^+\setminus S|\neq1$. Any such defined $J$ has the desired property: indeed,
\begin{align*}
J\intprod\sigma&=\sum_{\substack{\alpha\in S\\\beta\in R^+}}s_\alpha\langle J_\alpha,J_\beta\rangle\beta+\sum_{\alpha\in S}s_\alpha J_\alpha\intprod\sigma_0+\sum_{\beta\in R^+}\langle J',J_\beta\rangle\beta+J'\intprod\sigma_0.
\end{align*}
The first sum is equal to $\sum_{\alpha\in S}s_\alpha\alpha$ and thus vanishes. Second, $J_\alpha\intprod\sigma_0=0$, since $J_\alpha$ is supported on $\h_\alpha$ and $[\h_\alpha,\h_\alpha]\subset\t$. Third, $\langle J',J_\beta\rangle=0$ since $J'\h_\beta\subset\h_\beta^\perp$. And finally, since $R^+\setminus S$ is strongly orthogonal, $[\bigoplus_{\alpha\in R^+\setminus S}\h_{\alpha},\bigoplus_{\alpha\in R^+\setminus S}\h_{\alpha}]\subset\t$, so $J'\intprod\sigma_0=0$.

Well-balanced subsets always exist, and it follows from \cite[Thm.~3.2, Thm.~4.1]{roots} that one can find a well-balanced subset $S\subset R^+$ with $|R^+\setminus S|\neq1$ as long as $\h$ is not isomorphic to a product of Lie algebras of type $\R$ and $\su(2k+1)$, $k\geq1$, together with exactly one factor $\so(3)$ or $\so(5)$.

In the first of these two remaining cases, since $\h=\so(3)\oplus\R$ is excluded, we must have $\dim\t\geq4$. Then let $S$ be the set of positive roots of $\su(2k+1)$, which is automatically balanced by \cite[Thm.~3.2]{roots}. Let $(s_\alpha)_{\alpha\in S}$ be a corresponding choice of signs, let $\t_1\subset\t$ be any $3$-dimensional subspace orthogonal to $\so(3)\subset\h$, and let $\t_2:=\t\cap(\so(3)\oplus\t_1)^\perp$, so that
\[\h=\so(3)\oplus\t_1\oplus\t_2\oplus\bigoplus_{\alpha\in S}\h_\alpha.\]
Then we can declare an almost Hermitian structure
\[J:=J_1+J_2+\sum_{\alpha\in S}s_\alpha J_\alpha,\]
where $J_1$ is any almost Hermitian structure on $\so(3)\oplus\t_1$ with $J\so(3)\subset\t_1$, and $J_2$ is an arbitrary almost Hermitian structure on $\t_2$. Such $J$ satisfies $J\intprod\sigma=0$ for similar reasons as above: first we have $J_1\intprod\sigma=0$ since $[\so(3),\t_1]=0$, second $J_2\intprod\sigma=0$ because $[\t_2,\t_2]=0$, and third
$\sum_{\alpha\in S}s_\alpha J_\alpha\intprod\sigma=\sum_{\alpha\in S}s_\alpha\alpha=0$.

In the second case, let $S$ denote again the set of positive roots of $\su(2k+1)$, which is balanced. With the corresponding choice of signs $(s_\alpha)_{\alpha\in S}$ we consider $J$ of the form
\[J=J_\t+\sum_{\alpha\in S}s_\alpha J_\alpha+J',\]
where $J_\t$ is any almost Hermitian structure on $\t$, and $J'$ is an almost Hermitian structure on the root spaces of $\so(5)$. With the same arguments as above we see that $J\intprod\sigma=0$ if and only if $J'$ is such that $J'\intprod\sigma_{\so(5)}=0$. In the final section of the proof, we show that such a $J'$ exists.

To that end, assume without restriction that $\h=\so(5)$ is equipped with the invariant inner product given by $\langle X,Y\rangle=-\frac12\tr(XY)$. We denote with $E_{ij}$ the matrix with entry $1$ at position $(i,j)$ and zero everywhere else, and choose the following orthonormal basis of $\h$:
\begin{align*}
t_1&=E_{21}-E_{12},&
t_2&=E_{43}-E_{34},\\
X_{e_1}&=E_{15}-E_{51},&
Y_{e_1}&=E_{25}-E_{52},\\
X_{e_2}&=E_{35}-E_{53},&
Y_{e_2}&=E_{45}-E_{54},\\
X_{e_1+e_2}&=\frac{1}{\sqrt{2}}(E_{13}-E_{24}-E_{31}+E_{42}),&
Y_{e_1+e_2}&=\frac{1}{\sqrt{2}}(E_{14}+E_{23}-E_{32}-E_{41}),\\
X_{e_1-e_2}&=\frac{1}{\sqrt{2}}(E_{13}+E_{24}-E_{31}-E_{42}),&
Y_{e_1-e_2}&=\frac{1}{\sqrt{2}}(-E_{14}+E_{23}-E_{32}+E_{41}).
\end{align*}
Then $\t=\spann\{t_1,t_2\}$ is a Cartan subalgebra of $\h$. If $(e_1,e_2)$ denotes the dual basis in $\t^\ast$ of $\t$, we can pick $e_1,e_2$ as simple roots, so that $R^+(\so(5))=\{e_1,e_2,e_1+e_2,e_1-e_2\}$. For each $\alpha\in R^+$, the real root space is $\h_\alpha=\spann\{X_\alpha,Y_\alpha\}$, and we have
\begin{align*}
 \ad(t)X_\alpha&=\alpha(t)Y_\alpha,&
 \ad(t)Y_\alpha&=-\alpha(t)X_\alpha
\end{align*}
for all $t\in\t$. Thus we can write the canonical three-form as
\begin{align*}
 \sigma&=\sum_{\alpha\in R^+}\alpha\wedge X_\alpha\wedge Y_\alpha+\sigma_0,\\
 \sqrt{2}\sigma_0&=-X_{e_1}\wedge X_{e_2}\wedge(X_{e_1+e_2}+X_{e_1-e_2})+X_{e_1}\wedge Y_{e_2}\wedge(Y_{e_1-e_2}-Y_{e_1+e_2})\\
 &\qquad-Y_{e_1}\wedge X_{e_2}\wedge(Y_{e_1+e_2}+Y_{e_1-e_2})+Y_{e_1}\wedge Y_{e_2}\wedge(X_{e_1+e_2}-X_{e_1-e_2}).
\end{align*}
We make the following change of basis for $\t^\perp$:
\begin{align*}
 x_1&=X_{e_1},&
 x_2&=Y_{e_1},\\
 x_3&=X_{e_2},&
 x_4&=Y_{e_2},\\
 x_5&=-\frac{1}{\sqrt{2}}(X_{e_1+e_2}+X_{e_1-e_2}),&
 x_6&=\frac{1}{\sqrt{2}}(Y_{e_1-e_2}-Y_{e_1+e_2}),\\
 x_7&=-\frac{1}{\sqrt{2}}(Y_{e_1+e_2}+Y_{e_1-e_2}),&
 x_8&=\frac{1}{\sqrt{2}}(X_{e_1+e_2}-X_{e_1-e_2}).
\end{align*}
Then $\sigma$ simplifies to
\[\sigma=e_1\wedge(x^{12}+x^{57}+x^{68})+e_2\wedge(x^{34}+x^{56}+x^{78})+x^{135}+x^{146}+x^{237}+x^{248}\]
and one quickly sees that
\[J'=x^{17}+x^{26}+x^{38}+x^{45}\]
is a complex structure on $\t^\perp$ such that $J'\intprod\sigma=0$.
\end{proof}

We now move to the next Gray-Hervella class $\W_1\oplus\W_4$ and prove the following:
\begin{thm}
\label{thm:W1W4}
In dimension larger than $4$, every PSCT almost Hermitian structure in the class $\W_1\oplus\W_4$ is automatically Kähler.
\end{thm}
\begin{proof}
An almost Hermitian structure being of class $\W_1\oplus\W_4$ means the torsion is skew-symmetric with vanishing $\Lambda^{(2,1)+(1,2)}_0$-part. By \cite{GH}, this is equivalent to
\begin{equation}\label{gh}
(\nabla^g_XJ)(X)=\frac{-1}{2(m-1)}\left(|X|^2d^\ast \omega- \langle d^\ast \omega,X\rangle X-\langle d^\ast \omega,JX\rangle JX\right)
\end{equation}
for any $X\in TM$, where $m=\dim_\C M$ (thus $m\ge 3$ by hypothesis).

Assume first that $d^\ast \omega=0$, that is, $(M,g,J)$ is nearly Kähler (class $\W_1$). Since $\omega$ is $\nabla^\tau$-parallel, for every $X\in TM$ we have $\nabla^g_X\omega=-(\tau_X)_\ast\omega$, whence
\[0=X\intprod(\nabla^g_XJ)=-X\intprod(\tau_X)_\ast J=-[\tau_X,J]X=-\tau_XJX.\]
This shows that $\tau$ is of type $(0,3)+(3,0)$, so $\tau_X$ anti-commutes with $J$ for every $X\in TM$. On the other hand, if $(M,g,\tau)$ is a PSCT geometry we also have
\[0=JX\intprod(\tau_X)_\ast\tau=-\tau_{\tau_XJX}+[\tau_X,\tau_{JX}]=[\tau_X,\tau_{JX}],\]
showing that $\tau_X$ commutes with $\tau_{JX}=-J\tau_X$. But since $\tau_X$ also anti-commutes with $J$, this means that $\tau_X=0$ for every $X\in TM$, so $\nabla^g J=0$, i.e.~$(M,g,J)$ is Kähler.

We can thus assume for the remaining part of the proof that $d^\ast \omega\neq0$. Recall that $d^\ast \omega=2\omega\intprod\tau$ by \eqref{codiff}, so \eqref{gh} can be rewritten as
\begin{equation}
\tau_XJX=\frac{1}{m-1}(|X|^2\omega\intprod\tau-\langle\tau_X,\omega\rangle X-\langle\tau_{JX},\omega\rangle JX).
\label{W1W4}
\end{equation}
Consider the decomposition of $(M,g,\tau)$ given in Theorem \ref{thm:complex}. For any given $\alpha_i$-Sasaki factor $(S_i,g_i,\xi_i,\omega_i)$, let $X\in\xi_i^\perp\subset TS_i$ be a unit vector. Then $\tau_XJX=\pm\alpha_i\xi_i$ and $\tau_{\xi_i}=\pm\alpha_iX\wedge JX$. Plugging $X$ into \eqref{W1W4} and taking the inner product with $\xi_i$, we find
\[\pm(m-1)\alpha_i=\langle\omega\intprod\tau,\xi_i\rangle=\langle\tau_{\xi_i},\omega\rangle=\pm\alpha_i\omega(X,JX)=\pm\alpha_i,\]
which leads to a contradiction since $\alpha_i\neq0$. This shows that there is no $\alpha$-Sasaki factor in the decomposition of $M$.

Similarly, if $X$ is tangent to the Kähler factor $(N,g_N)$, then $\tau_X=0$, so \eqref{W1W4} reduces to $\omega\intprod\tau=0$, which contradicts the assumption $d^\ast\omega\neq0$. Consequently we can assume that $M=H$ for a Lie group $H$ with compact Lie algebra $\h$ of even rank, $g$ is a bi-invariant metric on $H$, $\tau=-\frac12\sigma$ with $\sigma$ the canonical three-form, and $J$ is some left-invariant almost Hermitian structure.

Write $H=\prod_iH_i\times\R^k$ with $H_i$ simple, and let $\eta:=\omega\intprod\tau$. If $k\neq0$, then for any nonzero $X\in\R^k$, \eqref{W1W4} implies that
\[|X|^2\eta=\langle\tau_{JX},\omega\rangle JX\]
since $\tau_X=-\frac12\ad(X)=0$. Thus $JX$ is proportional to $\eta$, which shows that $k\leq1$.

For any $X\in\h$, \eqref{W1W4} can be rewritten as
\begin{equation}\label{br}
[X,JX]=\frac{-2}{m-1}\left(|X|^2\eta-\langle\eta,X\rangle X-\langle\eta,JX\rangle JX\right),
\end{equation} 
and in particular
\begin{equation}\label{eta}
[\eta,J\eta]=0.
\end{equation} 
Polarizing \eqref{br}, we get for every $X,Y\in\h$:
\[[X,JY]+[Y,JX]=\frac{-2}{m-1}\left(2\langle X,Y\rangle\eta-\langle\eta,X\rangle Y-\langle\eta,Y\rangle X-\langle\eta,JX\rangle JY-\langle\eta,JY\rangle JX\right).\]
For $X=J\eta$, this reduces to
\[JY=-\frac{m-1}{2|\eta|^2}\left([J\eta,JY]+[\eta,Y]\right)+\frac{1}{|\eta|^2}\left(-\langle J\eta,Y\rangle\eta+\langle\eta,Y\rangle J\eta\right).\]
By \eqref{eta}, $J$, $\ad(\eta)$ and $\ad(J\eta)$ preserve $\spann\{\eta,J\eta\}$ and hence the complement $\{\eta,J\eta\}^\perp$. Let $J_0:=J|_{\{\eta,J\eta\}^\perp}$, $\ad_0(\eta):=\ad(\eta)|_{\{\eta,J\eta\}^\perp}$ and $\ad_0(J\eta):=\ad(J\eta)|_{\{\eta,J\eta\}^\perp}$. Then the above relation implies
\begin{equation}\label{j00}
J_0=-\frac{m-1}{2|\eta|^2}\left(\ad_0(\eta)+\ad_0(J\eta)\circ J_0\right).
\end{equation} 
In particular this shows that $\ad_0(J\eta)\circ J_0$ is skew-symmetric, so $\ad_0(J\eta)$ anticommutes with $J_0$. The relation \eqref{j00} can be equivalently written as
\begin{equation}\label{j0}J_0=-\frac{m-1}{2|\eta|^2}\left(\Id+\frac{m-1}{2|\eta|^2}\ad_0(J\eta)\right)^{-1}\circ\ad_0(\eta).\end{equation} 

Since by \eqref{br}, $\ad_0(\eta)$ commutes with $\ad_0(J\eta)$, \eqref{j0} shows that $\ad_0(\eta)$ commutes with $J_0$. Then \eqref{j00} implies that $J_0$ commutes with $\ad_0(J\eta)$. Since they also anti-commute, we must have $\ad_0(J\eta)=0$. In particular, $J\eta$ belongs to the center of $\h$. 

By assumption, $\dim(H)>4$, so $\rank\h\ge 2$. If the rank of $\h$ is at least $3$, then $\ad_0(\eta)$ must have a nontrivial kernel, which leads to a contradiction with \eqref{j0}.

Finally, the rank of $\h$ cannot be equal to $2$, since there are no compact Lie algebras of rank $2$ and dimension larger than $4$ with nontrivial center.
\end{proof}

Theorem~\ref{thm:W1W4} says the $\W_3$-component of the torsion of a PSCT almost Hermitian structure can never vanish in dimensions larger than $4$ unless the structure is Kähler. Recall that, by virtue of the torsion being skew-symmetric, the $\W_2$-component is always zero. Since PSCT almost Hermitian structures in the class $\W_3\oplus\W_4$ were classified in \cite{BPT} and those in $\W_1\oplus\W_3$ were described in Theorem~\ref{thm:W1W3}, we are left with the study of their intersection, i.e.~the class $\W_3$ of semi-Kähler and integrable PSCT almost Hermitian structures. It turns out that there are no nontrivial solutions in this case:

\begin{thm}
\label{thm:W3}
A PSCT almost Hermitian manifold of class $\W_3$ is Kähler.
\end{thm}
\begin{proof}
The statement only makes sense for $\dim(M)\ge 6$. By Theorem~\ref{thm:W1W3}, we know that $(M,g)$ is locally isometric to $(H,g_H)\times (N,g_N)$ where $H$ is a Lie group with compact Lie algebra $\h\neq\su(2)\oplus\R$ and bi-invariant metric $g_H$, $(N,g_N)$ is a Kähler manifold, and $\tau=-\frac12\sigma$, with $\sigma$ the canonical 3-form of $H$. Moreover,
\cite[Thm.~A]{BPT} shows that $J|_\h$ is a \emph{linear complex structure} in the sense of Samelson, i.e. it preserves a Cartan subalgebra $\t\subset\h$ and there is a choice of positive roots $R^+$ \cite{Pit,Sam} such that the (complex) root spaces corresponding to $R^+$ are contained in the eigenspace $\h^{1,0}$. Equivalently, there is a unique Borel subalgebra of $\h^\C$ containing $\h^{1,0}$. This means that the Samelson complex structure is given by
\[J|_\h=J_\t+\sum_{\alpha\in R^+}J_\alpha,\]
where $J_\t$ is some almost Hermitian structure on $\t$, and $J_\alpha$ are as in the proof of Theorem~\ref{thm:W1W3}. But then
\[J\intprod\tau=-\frac12 J|_\h\intprod\sigma=-\frac12\sum_{\alpha\in R^+}\alpha.\]
The sum of positive roots of $\h$ can only vanish if $H$ is abelian, thus $\tau=0$ and $(M,g,J)$ is Kähler.
\end{proof}

\section{Some other PSCT \texorpdfstring{$G$}{G}-structures}
\label{sec:other}

Using Corollary~\ref{corprod}, we can investigate PSCT geometries parallelizing $G$-structures for some other groups $G$. We begin with the Riemannian holonomy groups in the Berger list, other than $\U(n/2)$ (which was studied in the previous section), and $\mathrm{G}_2$ and $\Spin(7)$ which will be considered below.

\begin{thm}
\label{thm:parG}
Let $(M^n,g,\tau)$ be a simply connected PSCT geometry parallelizing a $G$-structure, where $G$ is either $\SU(k)$ with $n=2k\geq4$, or one of $\Sp(k)$, $\Sp(k)\Sp(1)$ with $n=4k\geq8$. Then $(M,g,\tau)$ decomposes as in \eqref{prod}, with $(M_i,g_i,\tau_i)$ being compact, simple Lie groups with flat $\nabla^{\tau_i}$.

Moreover, if $G=\SU(k)$, then the irreducible torsion-free Riemannian factors are Calabi--Yau. If $G=\Sp(k)$ or $\Sp(k)\Sp(1)$ and $\tau\neq0$, then the irreducible torsion-free Riemannian factors are hyperkähler.
\end{thm}
\begin{proof}
If $\nabla^\tau$ parallelizes a $G$-structure, then $\hol(\nabla^\tau)\subset\g$. The holonomy algebra is a diagonal sum of the holonomy algebras of the factors in \eqref{prod}, that is,  $\hol(\nabla^\tau)$ is a direct sum of subalgebras of $\so(n)$, each of the form $\diag(\pi(\h),0,\ldots,0)$, where $\h$ is either a compact simple Lie algebra and $\pi$ is its the adjoint representation, or $\h=\u(1)$ and $\pi$ is its standard representation on $\R^2$, or $(\pi,\h)$ is an irreducible Riemannian holonomy representation.

Suppose $(M,g,\tau)$ has a local irreducible factor $(M_i,g_i,\tau_i)$ as in Theorem~\ref{thm:irred}, and $\hol(\nabla^{\tau_i})\neq0$. First, since $\g$ does not contain any decomposable two-forms (i.e.~endomorphisms of rank $2$), $\hol(\nabla^\tau)$ cannot contain any $\u(1)$-summand, which rules our that $(M_i,g_i,\tau_i)$ is non-round Sasaki. Second, if $\hol(\nabla^{\tau_i})$ is a compact simple Lie algebra acting on $T_pM_i$ as its adjoint representation, then by Lemma~\ref{lem:reducedstructure}, it cannot preserve a $G$-structure on $T_pM$. Thus we must have $\hol(\nabla^{\tau_i})=0$, i.e.~$(M_i,g_i,\tau_i)$ is a compact simple Lie group with $(\pm)$-connection.

Let $(N_j,h_j)$ be a Riemannian factor in \eqref{prod}. If $G=\SU(k)$, then $\hol(\nabla^{h_j})\subset\u(l)$ by Lemma~\ref{lem:reducedstructure} (for $l=\dim_\C N_j$), and since
\[\diag(\u(l),0,\ldots,0)\cap\su(k)=\diag(\su(l),0,\ldots,0),\]
we can conclude $\hol(\nabla^{h_j})\subset\su(l)$. If $G=\Sp(k)$ or $\Sp(k)\Sp(1)$, we can again apply Lemma~\ref{lem:reducedstructure} to obtain $\hol(\nabla^{h_j})\subset\sp(l)$, $l=\dim_\H N_j$ (note that $r\neq0$ since $\tau\neq0$).
\end{proof}

We next consider the case $G=\mathrm{G}_2$:
\begin{thm}
\label{thm:g2}
Every PSCT geometry $(M^7,g,\tau)$ with $\tau\neq0$ parallelizing a $\rmG_2$-structure is locally isomorphic to one of $S^3\times\R^4$, or $S^3\times S^3\times\R$, or $S^3\times X^4$, where $X^4$ is a $K3$ surface. In each of these cases, $S^3$ is equipped with a bi-invariant metric and flat Cartan--Schouten connection, while the other factors are torsion-free.
\end{thm}
\begin{proof}
We proceed as in the proof of Theorem~\ref{thm:parG}. By Corollary \ref{corprod}, $(M,g,\tau)$ decomposes as in \eqref{prod}, and $\hol(\nabla^\tau)\subset\g_2\subset\so(7)$ is a diagonal sum of the holonomy algebras of the factors in \eqref{prod}. Suppose now that $\h$ is one of these holonomy algebras, embedded as $\diag(\pi(\h),0,\ldots,0)\subset\hol(\nabla^\tau)$.

First, we can rule out $\h=\u(1)$ since $\g_2\subset\so(7)$ does not contain any non-zero decomposable two-forms. To see this, assume that $X\wedge Y\in\g_2$ for some non-zero orthogonal vectors $X,Y\in\R^7$. Denoting by $\varphi\in\Lambda^3\R^7$ the $\rmG_2$-form, we have both $\varphi(X,Y)=0$ and
\[0=(X\wedge Y)_*\varphi=Y\wedge(X\intprod\varphi)-X\wedge(Y\intprod\varphi).\]
Taking the interior product with $Y$ in this relation and using the orthogonality of $X,Y$ yields $|Y|^2X\intprod\varphi=0$, which is impossible since $X,Y$ are non-zero and $\varphi$ has no kernel.

If $\diag(\pi(\h),0,\ldots,0)\subset\g_2$, then it is contained in a maximal proper subalgebra of $\g_2$. Up to conjugacy, these are the following:
\begin{itemize}
\item $\su(3)$, splitting $\R^7$ into the irreducible summands $\R\oplus\C^3$,
\item $\so(4)$, splitting $\R^7$ into the irreducible summands $\R^3\oplus\R^4$,
\item $\so(3)_{\mathrm{irr}}$, acting irreducibly on $\R^7$.
\end{itemize}
In the first case, $(\pi,\h)$ could be the Riemannian holonomy representation $(\C^3,\su(3))$, but then $\tau=0$. Any other simple or Riemannian holonomy algebra in $\su(3)$ is conjugate to $\su(2)$, splitting $\R^7=\R\oplus\R\oplus\R\oplus\C^2$. This is at the same time the only such subalgebra of $\so(4)$ for which $\R^7$ has a trivial subrepresentation. So we find the possibility $(\C^2,\su(2))$. The last case only has $\u(1)$-subalgebras, which we have already excluded.

Thus, either $\hol(\nabla^\tau)=\diag(\su(2),0,0,0)$ or $\hol(\nabla^\tau)=0$. In the first case, $(M,g,\tau)$ is locally the product of a K3 surface with some $3$-dimensional PSCT geometry with nonzero torsion and flat connection (so necessarily $S^3$). In the second case, $(M,g,\tau)$ is a $7$-dimensional PSCT geometry with $\tau\neq0$ and flat connection, thus locally isomorphic to either $S^3\times\R^4$ or $S^3\times S^3\times\R$.
\end{proof}

Finally, we study the case $G=\Spin(7)$:
\begin{thm}
\label{thm:spin7}
Let $(M^8,g,\tau)$ be a PSCT geometry with $\tau\neq0$, parallelizing a $\Spin(7)$-structure. Then it is locally isomorphic to one of $S^3\times\R^5$, or $S^3\times S^3\times\R^2$, or $\SU(3)$, or $S^3\times\R\times X^4$, where $X^4$ is a $K3$ surface. In each of these cases, $S^3$ resp.~$\SU(3)$ is equipped with a bi-invariant metric and flat Cartan--Schouten connection, while the other factors are torsion-free.
\end{thm}
\begin{proof}
As before, we apply Corollary \ref{corprod}  to $(M,g,\tau)$. Then $\hol(\nabla^\tau)\subset\spin(7)\subset\so(8)$ can be written as a diagonal sum of the holonomy algebras of the factors in~\eqref{prod}. Let $\h$ be one of these holonomy algebras. Since $\h$ is a diagonal summand of $\hol(\nabla^\tau)$, it either acts irreducibly on $\R^8$, or trivially on some nonzero subspace.

If $\h$ acts irreducibly, it has to be either a Riemannian holonomy algebra or a compact simple Lie algebra acting by its adjoint representation. The first case is not possible since it would imply that $\tau=0$. In the second case, $\h\cong\su(3)$ for dimensional reasons; however, the image of $\su(3)$ under its adjoint representation does not lie in (a conjugate of) $\spin(7)$, since $\spin(7)$ annihilates a nontrivial four-form on $\R^8$, but $\su(3)$ does not (by Chevalley--Eilenberg theory or a computation with LiE \cite{lie}, $(\Lambda^4\su(3))^{\su(3)}=0$). Thus this case is impossible as well.

We conclude that $\h$ must act reducibly on $\R^8$. This implies that $\h$ is contained in some conjugate of $\g_2\subset\spin(7)$, since $\Spin(7)$ acts transitively on the unit sphere in $\R^8$ and the stabilizer of every vector is $\rmG_2$. Then we can apply the arguments in the proof of Theorem~\ref{thm:g2} to see that either  $\hol(\nabla^\tau)$ is conjugate to $\diag(\su(2),0,\ldots,0)$, or $\hol(\nabla^\tau)=0$. Due to $\tau\neq0$, this means that $(M,g,\tau)$ is locally the product of a K3 surface with a flat torsionful $S^3\times\R$, or $(M,g,\tau)$ is locally a non-abelian compact Lie group with flat Cartan--Schouten connection, so either $S^3\times\R^5$, $S^3\times S^3\times\R^2$ or $\SU(3)$.
\end{proof}

For the purpose of the next theorem, we define a \emph{$\Spin(k)$-structure} on an oriented Riemannian manifold $(M^n,g)$ as a reduction of the orthonormal frame bundle to $\Spin(k)$, where $\Spin(k)$ is embedded into $\SO(n)$ by its \emph{spin representation} $\Sigma_k$, $n=\dim\Sigma_k$. By the \emph{spin representation} we mean the representation obtained by restricting to $\Spin(k)\subset\Cl^0_k$ the unique real irreducible representation of the real even Clifford algebra $\Cl^0_k$ for $k\not\equiv0\mod4$, or respectively the sum of the two irreducible $\Cl^0_k$-representations for $k\equiv0\mod4$. This means that for $k=0\mod 4$ the spin representation is the sum of two so-called half-spin representations $\Sigma^\pm_k$, cf.~\cite[\S11]{harvey}. If $k=8p+r$ with $r\in\{0,\ldots,7\}$, then $\dim\Sigma_k=2^{4p+a_r}$, where $a_0=a_1=0,\ a_2=1,\ a_3=2$, and $a_r=3$ for $4\leq r\leq 7$.

\begin{rem}
\label{rem:spin}
    Due to exceptional isomorphisms, for small $k$ the spin representation is in fact the standard representation of low-dimensional classical groups:
    \begin{itemize}
        \item $\Sigma_3$ is the standard representation of $\Spin(3)\cong\SU(2)$ on $\C^2$;
        \item $\Sigma_4^\pm$ is the composition of the projection to one of the factors of $\Spin(4)\cong\SU(2)\times\SU(2)$ with the standard representation of $\SU(2)$ on $\C^2$.
        \item $\Sigma_5$ is the standard representation of $\Spin(5)\cong\Sp(2)$ on $\H^2$;
        \item $\Sigma_6$ is the standard representation of $\Spin(6)\cong\SU(4)$ on $\C^4$;
    \end{itemize}
\end{rem}

We will thus be interested in $\Spin(k)$-structures mostly for $k\geq 7$. PSCT geometries with parallelizing a $\Spin(7)$-structure were treated by Theorem~\ref{thm:spin7}; the next theorem shows that for $k\geq8$, PSCT $\Spin(k)$-structures with nontrivial torsion are necessarily flat.

\begin{thm}\label{thm:spink}
Let $(M^n,g,\tau)$ be a  PSCT geometry parallelizing a $\Spin(k)$-structure, where $k\geq8$. Then
\begin{enumerate}[\upshape(i)]
    \item either $\tau=0$, $k=9$, and $(M,g)$ is locally isometric to the Cayley projective plane $\Oct\P^2$ or its noncompact dual $\Oct\H^2$, or
    \item $\hol(\nabla^\tau)=0$ and $(M,g,\tau)$ is locally isomorphic to a compact Lie group with a left-invariant $\Spin(k)$-structure.
\end{enumerate}
\end{thm}
\begin{proof}
Let $\rho:\Spin(k)\to\SO(n)$ denote the spin representation. One last time we apply Corollary~\ref{corprod} to $(M,g,\tau)$ and note that $\hol(\nabla^\tau)\subset\rho_\ast(\spin(k))\subset\so(n)$ is a diagonal sum of the holonomy algebras of the factors in~\eqref{prod}. The holonomy representations of each factor is either the adjoint representation $(\ad,\h)$ of a compact simple Lie algebra, or the standard representation of $\u(1)$ on $\R^2$, or an irreducible Riemannian holonomy representation. In fact, we can understand all of them as Riemannian holonomy representations: $(\ad,\h)$ is the holonomy representation of a bi-invariant metric on a compact simple Lie group with Lie algebra $\h$, and $(\R^2,\u(1))$ is the holonomy representation of a non-flat surface.

This shows that there always exists a Riemannian manifold $(M',g')$ with same holonomy algebra as $\nabla^\tau$, and by assuming that $M'$ is simply connected we can arrange $\Hol(\nabla^{g'})$ to be connected. It follows then from $\hol(\nabla^{g'})=\hol(\nabla^\tau)\subset\rho_\ast(\spin(k))$ that $\Hol(\nabla^{g'})\subset\rho(\Spin(k))$. In particular the orthonormal frame bundle of $(M',g')$ admits a reduction to $\Spin(k)$ and one can define on $(M',g')$ a parallel \emph{even Clifford structure} of rank $k$ in the sense of \cite{MS11}. By \cite[Thm.~2.14]{MS11}, either $\nabla^{g'}$ is flat, or $k\in\{8,9,10,12,16\}$ and $\Hol(\nabla^{g'})$ is listed in \cite[Table~2]{MS11}.

In the first case, it follows that $\hol(\nabla^\tau)=0$, and then Corollary~\ref{corprod} implies together with Theorem~\ref{thm:irred} implies that $(M,g,\tau)$ is locally isomorphic to a compact Lie group with $\nabla^\tau$ the $(-)$-connection.

For the remaining possibilities, we observe that the only of the holonomy groups in \cite[Table~2]{MS11} which is contained in $\Spin(k)$ is the one of the Cayley projective plane (and its non-compact dual) for $k=9$. (Note that for $k=16$, the holonomy of $\rmE_8/\Spin^+(16))$ is $\Spin(16)/{\Z_2}$ acting on $\Sigma_{16}^+$, which is only half of the spin representation $\Sigma_{16}$ as defined above.) Accordingly, $\hol(\nabla^\tau)=\spin(9)\subset\so(16)$, which means that $(M,g)$ must itself be locally isometric to $\Oct\P^2$ or $\Oct\H^2$.
\end{proof}

\section{Uniqueness of the characteristic connection}
\label{sec:unique}

In order to justify talking about \emph{the torsion three-form} of a $G$-structure admitting a $G$-connection with skew-symmetric torsion (i.e.~a \emph{characteristic connection}), this connection should be unique. Recall from the introduction that this is known to be the case for almost Hermitian structures, almost contact metric structures, $\rmG_2$-structures, $\Spin(7)$-structures and $\Spin(9)$-structures.

In this section we will show that on a given Riemannian manifold $(M^n,g)$, a characteristic connection for $G=\U(n/2)$ or $G=\Sp(n/4)\Sp(1)$ is unique if it exists. As mentioned above, for $\U(n/2)$ the result is well known, but since the proof is very short, we give it here for convenience. The claim is a direct consequence of the following.

\begin{lem}
\label{lem:uniqueskew}
\begin{enumerate}[\upshape(i)]
\item If $(V^{2m},g,J)$ is a Hermitian vector space and $\tau\in\Lambda^3V$ satisfies $\tau_X\in \u(m)=\Lambda^{1,1}V$ for every $X\in V$, then $\tau=0$.
\item If $(V^{4q},g,J_1,J_2,J_3)$ is a quaternionic Hermitian vector space, $q\ge 2$, and $\tau\in\Lambda^3V$ satisfies $\tau_X\in \sp(q)\sp(1)$ for every $X\in V$, then $\tau=0$.
\end{enumerate}
\end{lem}
\begin{proof}
\begin{enumerate}[\upshape(i)]
\item For every $X,Y,Z\in V$ we have by assumption $\tau(X,Y,Z)=\tau(X,JY,JZ)$. Using the skew-symmetry of $\tau$ we obtain:
\begin{align*}
\tau(X,Y,Z)&=\tau(X,JY,JZ)=-\tau(JY,X,JZ)=\tau(JY,JX,Z)\\
&=-\tau(Z,JX,JY)=-\tau(Z,X,Y)=-\tau(X,Y,Z),
\end{align*}
thus showing that $\tau=0$.
\item Denoting by $\omega_i$ the Hermitian 2-form $g(J_i,\cdot,\cdot)$ for $i=1,2,3$, the hypothesis can be equivalently stated as follows: there exists $1$-forms $\alpha_1,\alpha_2,\alpha_3$ and $\sigma\in V^*\otimes\sp(q)$ such that 
\[\tau_X=\sigma_X+\sum_{i=1}^3\alpha_i(X)\omega_i,\qquad\forall X\in V.\]
The condition $\sigma_X\in\sp(q)$ means that $\sigma_X(J_iY,J_iZ)=\sigma_X(Y,Z)$ for all $Y,Z\in V$. Since $\omega_i(J_1Y,J_1Z)=-\omega_i(Y,Z)$ for $i=2,3$, we obtain:
\begin{align*}
\tau(X,J_1Y,J_1Z)&=\sigma_X(Y,Z)+\sum_{i=1}^3\alpha_i(X)\omega_i(J_1Y,J_1Z)\\
&=\tau(X,Y,Z)-2\sum_{i=2}^3\alpha_i(X)\omega_i(Y,Z)
\end{align*}
On the other hand, using this relation again but with different arguments yields:
\begin{align*}
\tau(X,J_1Y,J_1Z)&=- \tau(J_1Y,X,J_1Z)=\tau(J_1Y,J_1X,Z)-2\sum_{i=2}^3\alpha_i(J_1Y)\omega_i(J_1X,Z)\\
&=-\tau(Z,J_1X,J_1Y)-2\sum_{i=2}^3\alpha_i(J_1Y)\omega_i(J_1X,Z)\\
&=-\tau(Z,X,Y)+2\sum_{i=2}^3\alpha_i(Z)\omega_i(X,Y)-2\sum_{i=2}^3\alpha_i(J_1Y)\omega_i(J_1X,Z).
\end{align*}
Comparing the above two equations shows that
\begin{equation}\label{aj}
\tau(X,Y,Z)=\sum_{i=2}^3(\alpha_i(X)\omega_i(Y,Z)+\alpha_i(Z)\omega_i(X,Y)-\alpha_i(J_1Y)\omega_i(J_1X,Z))
\end{equation}
for every $X,Y,Z\in V$. For $X=Y$ this relation gives 
\[0=\alpha_2(X)J_2(X)+\alpha_3(X)J_3(X)+\alpha_2(J_1X)J_3(X)-\alpha_3(J_1X)J_2(X),\]
thus showing that $\alpha_3(J_1X)=\alpha_2(X)$ and $\alpha_2(J_1X)=-\alpha_3(X)$. Substituting in \eqref{aj} yields
\[\tau=\alpha_2\wedge\omega_2+\alpha_3\wedge\omega_3.\]
By replacing $J_1$ with $J_2$ and $J_3$ in the above argument, we obtain that $\alpha_1\wedge\omega_1=\alpha_2\wedge\omega_2=\alpha_3\wedge\omega_3$. 

We claim that this implies $\alpha_1=\alpha_2=\alpha_3=0$ for $q\ge 2$. Indeed, taking the Lefschetz operator associated to $J_1$ we get after straightforward calculations $J_1\intprod (\alpha_1\wedge\omega_1)=(2q-1)\alpha_1$ and $J_1\intprod (\alpha_2\wedge\omega_2)=-J_3\alpha_2$. In particular we have $g(\alpha_2,\alpha_2)=(2q-1)^2g(\alpha_1,\alpha_1)$. By circular permutation of the indices we also obtain $g(\alpha_3,\alpha_3)=(2q-1)^2g(\alpha_2,\alpha_2)$ and $g(\alpha_1,\alpha_1)=(2q-1)^2g(\alpha_3,\alpha_3)$, so $g(\alpha_i,\alpha_i)=(2q-1)^6g(\alpha_i,\alpha_i)$ for $i=1,2,3$, thus proving our claim since $(2q-1)>1$. We thus have shown that $\tau=0$ thanks to \eqref{aj}.
\end{enumerate}
\end{proof}

\begin{cor}
Let $G=\U(n/2)$ or $G=\Sp(n/4)\Sp(1)$. If a $G$-structure on a Riemannian manifold $(M^n,g)$ admits a characteristic connection, then this connection is unique.
\end{cor}
\begin{proof}
The $G$-connections on $(M,g)$ form an affine space modeled on $\Omega^1(M,\g)$. Likewise, the affine space of metric connections with skew-symmetric torsion is modeled on $\Omega^3(M)$. Lemma~\ref{lem:uniqueskew} now shows that $\Omega^1(M,\g)\cap\Omega^3(M)=0$.
\end{proof}

\end{document}